\documentclass{birkjour}

\usepackage{amssymb,amsmath,newlfont,enumerate}

\theoremstyle{plain}
\newtheorem{theorem}{Theorem}[section]
\newtheorem{lemma}[theorem]{Lemma}

\theoremstyle{definition}
\newtheorem{definition}[theorem]{Definition}

\theoremstyle{remark}

 \numberwithin{equation}{section}

\newcommand{\DD}{{\mathbb D}}
\newcommand{\RR}{{\mathbb R}}
\newcommand{\TT}{{\mathbb T}}
\newcommand{\cH}{{\cal H}}

\renewcommand{\hat}{\widehat}

\begin{document}

\title[Outer functions and divergence in de Branges--Rovnyak spaces]{Outer functions and divergence\\ in de Branges--Rovnyak spaces}

\author[Mashreghi]{Javad Mashreghi}
\address{D\'epartement de math\'ematiques et de statistique, 
Universit\'e Laval, \\Qu\'ebec (QC), Canada G1V 0A6}
\email{javad.mashreghi@mat.ulaval.ca} 

\author[Ransford]{Thomas Ransford}
\address{D\'epartement de math\'ematiques et de statistique, 
Universit\'e Laval, \\Qu\'ebec (QC), Canada G1V 0A6}
\email{thomas.ransford@mat.ulaval.ca}

\thanks{Mashreghi supported by a grant from NSERC\\
Ransford supported by grants from NSERC  
and the Canada Research Chairs program}

\begin{abstract}
In most classical holomorphic function spaces on the unit disk
in which the polynomials are dense, 
a function $f$ can be approximated in  norm  
by its dilates $f_r(z):=f(rz)~(r<1)$,
in other words, $\lim_{r\to1^-}\|f_r-f\|=0$.
We construct a de Branges--Rovnyak space $\cH(b)$ 
in which the polynomials are dense, and a function $f\in\cH(b)$ such that  
 $\lim_{r\to1^-}\|f_r\|_{\cH(b)}=\infty$.
The essential feature of our construction lies in the fact  that $b$ is an outer function.
\end{abstract}

\subjclass{46E22, 47B32, 30H15}

\keywords{De Branges--Rovnyak space, outer function, Toeplitz operator}

\maketitle

\section{Introduction}\label{S:intro}

In most holomorphic function spaces on the unit disk, at least those in which the polynomials are dense, the radial dilates  of a function in the space converge to the function in the norm of the space. In other words, writing $f_r(z):=f(rz)$, we have $\lim_{r\to1^-}\|f_r-f\|=0$ for all $f$ in the space.

However, perhaps surprisingly, this is not always true.
An example was given in \cite{EFKMR16} where in fact $\lim_{r\to1^-}\|f_r\|=\infty$.
The space in question was a 
de Branges--Rovnyak space $\cH(b)$.
We shall give the precise definition of $\cH(b)$ in \S\ref{S:background}. Suffice it to say
that de Branges--Rovnyak spaces  
are a family of subspaces $\cH(b)$ of the Hardy space~$H^2$,
parametrized by elements $b$ of the unit ball of $H^\infty$. In the example constructed in \cite{EFKMR16},
the function $b$ was the product of a rational function and a carefully chosen Blaschke product, and the justification of the construction depended heavily on certain properties Blaschke products with uniformly separated zeros. 

Our purpose in this article is to show that such an example can also be constructed with $b$ being an outer function. This has the advantage that, as we no longer need to cite results about Blaschke products with uniformly separated zeros, the construction is more elementary than that given in \cite{EFKMR16}. 

\begin{theorem}\label{T:main}
There exist an outer function $b$ in the unit ball of $H^\infty$ and a function $f\in\cH(b)$ such that polynomials are dense in $\cH(b)$, yet 
\begin{equation}\label{E:main}
\lim_{r\to1^-}\|f_r\|_{\cH(b)}=\infty.
\end{equation}
\end{theorem}

Our proof proceeds via an auxiliary result of independent interest, in which we construct an outer function with very precise control over its behavior along the radius $(0,1)$.
The necessary background on de Branges--Rovnyak spaces is described in \S\ref{S:background}. Theorem~\ref{T:main} is proved in \S\ref{S:proof}, and we make some concluding remarks in \S\ref{S:remarks}.

\section{Background on $\cH(b)$-spaces}\label{S:background}

We denote by $\DD$ and $\TT$ the open unit disk and the unit circle respectively. 
Also, we denote by $H^2$  the Hardy space on $\DD$.
Given $\psi\in L^\infty(\TT)$, 
the corresponding Toeplitz operator $T_\psi:H^2\to H^2$ 
is defined by
\[
T_\psi f:=P_+(\psi f) \qquad(f\in H^2),
\]
where $P_+:L^2(\TT)\to H^2$ denotes 
the orthogonal projection of $L^2(\TT)$ onto $H^2$. 
Clearly $T_\psi$ is a bounded operator on $H^2$ 
with $\|T_\psi\|\le\|\psi\|_{L^\infty(\TT)}$. 
If $h\in H^\infty$, 
then $T_h$ is simply the operator of multiplication by $h$ 
and its adjoint is $T_{\overline{h}}$.

\begin{definition}
Let $b\in H^\infty$
with $\|b\|_{H^\infty}\le1$.
The associated \emph{de Branges--Rovnyak space} 
$\cH(b)$ is the image of $H^2$ 
under the operator $(I-T_ bT_{\overline b})^{1/2}$.
We define a norm  on $\cH(b)$ making $(I-T_bT_{\overline{b}})^{1/2}$ 
a partial isometry from $H^2$ onto $\cH(b)$, namely
\[
\|(I-T_bT_{\overline{b}})^{1/2}f\|_{\cH(b)}:=\|f\|_{H^2}
\qquad(f\in H^2\ominus \ker(I-T_bT_{\overline{b}})^{1/2})).
\]
\end{definition}

This is the definition of $\cH(b)$ as given in \cite{Sa94}. 
The original definition of de Branges and Rovnyak,
based on the notion of complementary space, 
is different but equivalent. 
An explanation of the equivalence can be found in \cite[pp.7--8]{Sa94}. 
A third approach is to start from the positive kernel
\[
k_w^b(z):=\frac{1-\overline{b(w)}b(z)}{1-\overline{w}z}
\qquad(z,w\in\DD),
\]
and to define $\cH(b)$ as 
the reproducing kernel Hilbert space associated with this kernel.
For a more detailed description of these spaces, 
we refer to the recent two-volume work \cite{FM16a,FM16b}.

The general theory of $\cH(b)$-spaces splits  into two cases, 
according to whether $b$ is an extreme point or a non-extreme point 
of the unit ball of $H^\infty$. 
This dichotomy is illustrated by following result.

\begin{theorem}\label{T:nonextreme}
Let $b\in H^\infty$ with $\|b\|_{H^\infty}\le 1$. 
The following are equivalent:
\begin{enumerate}[(i)]
\item $b$ is a non-extreme point of the unit ball of $H^\infty$;
\item  $\log (1-|b|^2)\in L^1(\TT)$;
\item  $\cH(b)$ contains all functions holomorphic in a neighborhood of $\overline{\DD}$;
\item polynomials are dense in $\cH(b)$.
\end{enumerate}
\end{theorem}

\begin{proof}
The equivalence between (i) and (ii) is proved in \cite[Theorem~7.9]{Du70}. 
The equivalence between (i) and (iii) follows from \cite[\S IV-6 and \S V-1]{Sa94}.
Finally, the equivalence between (i) and (iv) follows from \cite[\S IV-2, \S IV-3 and \S V-1]{Sa94}. (Another, more constructive, proof of the density of polynomials in the case when $b$ is non-extreme can be found in \cite[Theorem~5.1]{EFKMR16}.)
\end{proof}

Henceforth we shall simply say that $b$ is `extreme' or `non-extreme', 
it being understood that this is relative to the unit ball of $H^\infty$.

From the equivalence between (i) and (ii), it follows that, 
if  $b$ is non-extreme,
then there is an  outer function $a$ 
such that $a(0)>0$ and  $|a|^2+|b|^2=1$ a.e.\ on $\TT$ 
(see \cite[\S IV-1]{Sa94}). 
The function $a$ is uniquely determined by $b$. 
We shall call $(b,a)$ a \emph{pair}. 
The following result gives 
a useful characterization of $\cH(b)$ in this case.

\begin{theorem}[\protect{\cite[\S IV-1]{Sa94}}]\label{T:f+}
Let $b$ be non-extreme, 
let $(b,a)$ be a pair and let $f\in H^2$.
Then $f\in\cH(b)$ if and only if 
$T_{\overline b}f\in T_{\overline a}(H^2)$.
In this case, there exists a unique function $f^+\in H^2$
such that $T_{\overline b}f=T_{\overline a}f^+$, and
\begin{equation}\label{E:f+}
\|f\|_{\cH(b)}^2=\|f\|_{H^2}^2+\|f^+\|_{H^2}^2.
\end{equation}
\end{theorem}

\section{Proof of Theorem~\ref{T:main}}\label{S:proof}

We shall prove Theorem~\ref{T:main} by establishing  the following slightly stronger result.

\begin{theorem}\label{T:stronger}
There exist an outer function $b$ non-extreme in the unit ball of $H^\infty$
and a function $f\in\cH(b)$
such that 
\begin{equation}\label{E:stronger}
\lim_{r\to 1^-}|(f_r)^+(0)|=\infty
\end{equation}
\end{theorem}

By Theorem~\ref{T:nonextreme}\,(iv), if $b$ is non-extreme, then polynomials are dense in $\cH(b)$.
Also,  if $b$ is non-extreme and $f\in\cH(b)$, then
by Theorem~\ref{T:nonextreme}\,(iii) we have $f_r\in\cH(b)$ for all $r\in(0,1)$, and
from Theorem~\ref{T:f+} we obtain
\[
\|f_r\|_{\cH(b)}\ge \|(f_r)^+\|_{H^2}\ge|(f_r)^+(0)|.
\]
Thus \eqref{E:stronger} implies \eqref{E:main},
and Theorem~\ref{T:main} is indeed a consequence of Theorem~\ref{T:stronger}.

In what follows, we shall write
$k_w(z):=1/(1-\overline{w}z)$, the Cauchy kernel.
It is the reproducing kernel for $H^2$ in the sense that
$f(w)=\langle f,k_w\rangle_{H^2}$ 
for all $f\in H^2$ and $w\in\DD$. 
In particular,
$\|k_w\|_{H^2}^2=\langle k_w,k_w\rangle_{H^2}=k_w(w)=1/(1-|w|^2)$.
We remark that $k_w$ has the useful property that
$T_{\overline{h}}(k_w)=\overline{h(w)}k_w$ for all $h\in H^\infty$. 
Indeed, given $g\in H^2$, we have
\[
\langle g, T_{\overline{h}}(k_w)\rangle_{H^2}
=\langle hg,k_w\rangle_{H^2}
=h(w)g(w)
=h(w)\langle g,k_w\rangle_{H^2}
=\langle g,\overline{h(w)}k_w\rangle_{H^2}.
\]

The proof of Theorem~\ref{T:stronger} depends on two lemmas.
The first lemma, a slight generalization of a result in \cite{EFKMR16}, 
provides a class of functions $f$ 
for which $f^+$ is readily identifiable.

\begin{lemma}\label{L:sumcjkwj}
Let $b$ be  non-extreme, 
let $(b,a)$ be a pair and let $\phi:=b/a$. Let
\[
f:=\sum_{j\ge1} c_j k_{w_j},
\]
where $(w_j)_{j\ge1}$ is a sequence in $\DD$, and  
$(c_j)_{j\ge1}$ are scalars satisfying
\begin{equation}\label{E:ccond} 
\sum_{j\ge1}|c_j|(1+|\phi(w_j)|)(1-|w_j|)^{-1/2}<\infty.
\end{equation}
Then $f\in\cH(b)$ and
\begin{equation}\label{E:sum}
f^+=\sum_{j\ge1}c_j\overline{\phi(w_j)}k_{w_j}.
\end{equation}
\end{lemma}

\begin{proof}
Note first of all that
\[
\sum_{j\ge1}|c_j|\|k_{w_j}\|_{H^2}
=\sum_{j\ge1}|c_j|(1-|w_j|^2)^{-1/2}
<\infty,
\]
so the series defining $f$ converges absolutely in $H^2$. Further, defining
\[
g:=\sum_{j\ge1}c_j\overline{\phi(w_j)}k_{w_j},
\]
we likewise have $g\in H^2$ and
\begin{align*}
T_{\overline{a}}(g)
&=\sum_{j\ge1} c_j\overline{\phi(w_j)} T_{\overline{a}}(k_{w_j})
=\sum_{j\ge1} c_j\overline{\phi(w_j)}\,\overline{a(w_j)}k_{w_j}\\
&=\sum_{j\ge1} c_j\overline{b(w_j)}k_{w_j}
=\sum_{j\ge1} c_j T_{\overline{b}}(k_{w_j})
=T_{\overline{b}}(f).
\end{align*}
By Theorem~\ref{T:f+}, it follows that $f\in\cH(b)$, and $f^+=g$.
\end{proof}

The second lemma, inspired by an idea in \cite{CGR10},
yields an outer function with very precise control over
the growth along the radius $[0,1)$.

\begin{lemma}\label{L:outer}
Let $(w_n)$ be a strictly increasing sequence in $(0,1)$ such that $w_n\to1$,
and let $(\rho_n)$ be a positive sequence such that
\begin{equation}\label{E:rhohyp}
\sum_{k>n}\Bigl(\frac{1-w_{k}}{1-w_{k+1}}\Bigr)\rho_k=o(\rho_n)
\qquad(n\to\infty).
\end{equation}
Then there exists an outer function $\phi$ 
such that $|\phi|\ge1$ on $\DD$ and, for all $n$,
\begin{equation}\label{E:outer}
\inf_{r\in[w_n,w_{n+1}]}\log\Bigl|\frac{\phi(rw_n)}{\phi(w_n)}\Bigr|\ge \frac{\rho_n}{1-w_n}.
\end{equation}
\end{lemma}

\begin{proof}
Let $\Phi$ be the outer function on the upper half-plane such that
\[
\log|\Phi|=\sum_{k\ge1}(\epsilon_k/t_k)1_{[2t_k,3t_k]}
\quad\text{a.e. on~}\RR,
\]
where
\[
t_k:=\frac{1-w_k^2}{1+w_k^2}
\quad\text{and}\quad
\epsilon_k:=\Bigl(\frac{1-w_{k}}{1-w_{k+1}}\Bigr)\rho_k.
\]
Note that
\[
\frac{1}{\pi}\int_\RR\frac{\log|\Phi(x)|}{1+x^2}\,dx
\le \sum_{k\ge1}\epsilon_k=
\sum_{k\ge1}\Bigl(\frac{1-w_{k}}{1-w_{k+1}}\Bigr)\rho_k<\infty,
\]
so $\Phi$ is well-defined. Also, since $\log|\Phi|\ge0$ a.e.\ on $\RR$, it follows that $|\Phi|\ge1$ on the upper half-plane. 

Define $\phi$ on the unit disk by
\[
\phi(z):=\Phi\Bigl(i\frac{1-z}{1+z}\Bigr)\quad(|z|<1).
\]
Then $\phi$ is also  outer  and $|\phi|\ge1$.
Fix $n$ and let $r\in[w_n,w_{n+1}]$. Then, writing
\[
u_n:=\frac{1-rw_n}{1+rw_n}
\quad\text{and}\quad
v_n:=\frac{1-w_n}{1+w_n},
\]
we have
\begin{align*}
\log\Bigl|\frac{\phi(rw_n)}{\phi(w_n)}\Bigr|
&=\log|\Phi(iu_n)|-\log|\Phi(iv_n)|\\
&=\frac{1}{\pi}\int_\RR \Bigl(\frac{u_n}{u_n^2+x^2}-\frac{v_n}{v_n^2+x^2}\Bigr)\log|\Phi(x)|\,dx\\
&=\frac{1}{\pi}\sum_{k\ge1}\frac{\epsilon_k}{t_k}\int_{2t_k}^{3t_k}\Bigl(\frac{u_n}{u_n^2+x^2}-\frac{v_n}{v_n^2+x^2}\Bigr)\,dx.
\end{align*}
(The intervals $[2t_k,3t_k]$ are not necessarily disjoint, but this does not matter.)
Note that $w_n^2\le rw_n\le w_n$, so $v_n\le u_n\le t_n$.
Thus, if $x\ge t_n$, then
\[
\frac{u_n}{u_n^2+x^2}-\frac{v_n}{v_n^2+x^2}
=\frac{(u_n-v_n)(x^2-u_nv_n)}{(u_n^2+x^2)(v_n^2+x^2)}\ge0,
\]
whence, for all $k< n$, we have
\[
\frac{1}{t_k}\int_{2t_k}^{3t_k}\Bigl(\frac{u_n}{u_n^2+x^2}-\frac{v_n}{v_n^2+x^2}\Bigr)\,dx\ge0.
\]
The same is true if $k=n$, but in this case we have a better estimate. Indeed, if $x\in[2t_n,3t_n]$, then
\[
\frac{u_n}{u_n^2+x^2}-\frac{v_n}{v_n^2+x^2}
=\frac{(u_n-v_n)(x^2-u_nv_n)}{(u_n^2+x^2)(v_n^2+x^2)}
\ge\frac{(u_n-v_n)(3t_n^2)}{(10t_n^2)(10t_n^2)}=\frac{3(u_n-v_n)}{100t_n^2},
\]
so
\[
\frac{1}{t_n}\int_{2t_n}^{3t_n}\Bigl(\frac{u_n}{u_n^2+x^2}-\frac{v_n}{v_n^2+x^2}\Bigr)\,dx\ge\frac{3(u_n-v_n)}{100t_n^2}.
\]
Lastly, for all $k>n$, we clearly have
\[
\frac{1}{t_k}\int_{2t_k}^{3t_k}\Bigl(\frac{u_n}{u_n^2+x^2}-\frac{v_n}{v_n^2+x^2}\Bigr)\,dx\ge
-\frac{1}{t_k}\int_{2t_k}^{3t_k}\frac{v_n}{v_n^2+x^2}\,dx\ge
-\frac{1}{v_n}.
\]
Putting these estimates together, we arrive at the inequality
\[
\log\Bigl|\frac{\phi(rw_n)}{\phi(w_n)}\Bigr|
\ge \frac{1}{\pi}\Bigl(\epsilon_n\frac{3(u_n-v_n)}{100t_n^2}-\Bigl(\sum_{k>n}\epsilon_k\Bigr)\frac{1}{v_n}\Bigr).
\]
Now, a simple calculation gives
\begin{align*}
u_n-v_n
&=\frac{1-rw_n}{1+rw_n}-\frac{1-w_n}{1+w_n}
=\frac{2(1-r)w_n}{(1+rw_n)(1+w_n)}\\
&\ge \frac{2(1-r)w_1}{4}\ge \frac{w_1}{2}(1-w_{n+1}).
\end{align*}
Also, we have
\[
t_n
=\frac{1-w_n^2}{1+w_n^2}
\le 2(1-w_n)
\quad\text{and}\quad
v_n=\frac{1-w_n}{1+w_n}\ge(1-w_n)/2.
\]
It follows that there are constants $C_1,C_2>0$ such that
\begin{align*}
\log\Bigl|\frac{\phi(rw_n)}{\phi(w_n)}\Bigr|
&\ge C_1\epsilon_n\frac{(1-w_{n+1})}{(1-w_n)^2}-C_2\Bigl(\sum_{k>n}\epsilon_k\Bigr)\frac{1}{(1-w_n)}\\
&=\frac{1}{1-w_n}\Bigl(C_1\rho_n-C_2\sum_{k>n}\frac{1-w_{k}}{1-w_{k+1}}\rho_k\Bigr).
\end{align*}
The hypothesis \eqref{E:rhohyp} on the sequence $(\rho_n)$ implies that the term in parentheses is at least $C_1\rho_n/2$ for all large enough $n$. Thus
\[
\log\Bigl|\frac{\phi(rw_n)}{\phi(w_n)}\Bigr|\ge (C_1/2)\frac{\rho_n}{1-w_n}
\]
for all $r\in[w_n,w_{n+1}]$ and all sufficiently large $n$.
Replacing $\phi$ by a large enough power of itself, we may ensure that this inequality holds for all $n$
and with a constant $C_1/2=1$. Thus \eqref{E:outer} holds.
\end{proof}

\begin{proof}[Proof of Theorem~\ref{T:stronger}]
Fix $\alpha,\beta$ with $1<\alpha<\beta<\alpha+1$, and define
\[
w_n:=1-e^{-n^\beta}
\quad\text{and}\quad
\rho_n:=e^{-n^\alpha}.
\]
Then
\[
\sum_{k>n}\Bigl(\frac{1-w_{k}}{1-w_{k+1}}\Bigr)\rho_k
=\sum_{k>n}e^{(-k^\beta+(k+1)^\beta-k^\alpha)}
=\sum_{k>n}e^{-k^\alpha(1+o(1))}=o(e^{-n^\alpha}),
\]
so condition \eqref{E:rhohyp} is satisfied.
By Lemma~\ref{L:outer}, there is an outer function $\phi$ such that $|\phi|\ge1$ on $\DD$ and
\begin{equation}\label{E:lowerbd}
\log\Bigl|\frac{\phi(rw_n)}{\phi(w_n)}\Bigr|\ge e^{n^\beta-n^\alpha}
\quad(r\in[w_n,w_{n+1}],~n\ge1).
\end{equation}
Replacing $\phi(z)$ by $\phi(z)\overline{\phi(\overline{z})}$,
we may further assume that $\phi$ takes positive real values on $(-1,1)$.

Let $a,b$ be the outer functions (normalized to be positive at $0$) satisfying 
\[
|a|^2=\frac{1}{1+|\phi|^2}
\quad\text{and}\quad
|b|^2=\frac{|\phi|^2}{1+|\phi|^2}
\quad\text{a.e.\ on~}\TT.
\]
Then $(b,a)$ is a pair and $b/a=\phi$. In particular, as $\log(1-|b|^2)=2\log|a|\in L^1(\TT)$,
the function $b$ is a non-extreme point of the unit ball of $H^\infty$, by Theorem~\ref{T:nonextreme}.

Define
\[
f:=\sum_{j\ge1}c_jk_{w_j},
\]
where $c_j:=(1-w_j)^{1/2}/(j^2\phi(w_j))$. 
Clearly \eqref{E:ccond} is satisfied, so, by Lemma~\ref{L:sumcjkwj}, we have $f\in\cH(b)$.
Also, for each $r\in(0,1)$, we have
$f_r=\sum_{j\ge1}c_jk_{rw_j}$,
and a second application of Lemma~\ref{L:sumcjkwj} shows that $f_r\in\cH(b)$ with
$(f_r)^+=\sum_{j\ge1}c_j\phi(rw_j)k_{rw_j}$.
In particular,
\begin{equation}\label{E:fr+}
(f_r)^+(0)=\sum_{j\ge1}c_j\phi(rw_j) \quad(0<r<1).
\end{equation}
Now, given $r\in(w_1,1)$, choose $n$ so that $r\in[w_n,w_{n+1}]$. 
Then, since all the terms in the series \eqref{E:fr+} are positive, we have
\[
\sum_{j\ge1}c_j\phi(rw_j)\ge c_n\phi(rw_n)
=\frac{(1-w_n)^{1/2}\phi(rw_n)}{n^2\phi(w_n)}
\ge \frac{e^{-n^\beta/2}\exp(e^{n^\beta-n^\alpha})}{n^2},
\]
where the last inequality comes from \eqref{E:lowerbd}.
Clearly the right-hand side tends to infinity as $n\to\infty$. 
Hence $(f_r)^+(0)\to\infty$ as $r\to1^-$.
\end{proof}

\section{Concluding remarks}\label{S:remarks}

\subsection{Rate of growth}
 The construction above yields a function $f$ satisfying
\[
\log\|f_r\|_{\cH(b)}\ge \frac{c}{1-r}\exp\Bigl(-\Bigl(\log\frac{1}{1-r}\Bigr)^{\alpha/\beta}\Bigr)
\quad(0<r<1),
\]
where $c$ is a positive constant. Whilst this estimate could be slightly improved with a more careful choice of $(w_n)$ and $(\rho_n)$, it is not far from optimal. Indeed, if $b$ is non-extreme and $f\in\cH(b)$, then necessarily
 $\log^+\|f_r\|_{\cH(b)}=o(1/(1-r))$ as $r\to1^-$
(see \cite[Theorem~5.2]{CGR10}).

\subsection{Failure of Sarason's formula}
Let $b$ be non-extreme, 
let $(b,a)$ be a pair 
and let $\phi:=b/a$, 
say $\phi(z)=\sum_{j\ge0}\hat{\phi}(j)z^j$. 
It was shown by Sarason in \cite{Sa08} that,
if $f$ is holomorphic in a neighborhood of $\overline{\DD}$,
say $f(z)=\sum_{k\ge0}\hat{f}(k)z^k$, 
then the series 
$\sum_{j\ge0}\hat{f}(j+k)\overline{\hat{\phi}(j)}$ 
converges absolutely for each $k$, and 
\begin{equation}\label{E:hbnorm}
\|f\|_{\cH(b)}^2
=\sum_{k\ge0}|\hat{f}(k)|^2
+\sum_{k\ge0}\Bigl|\sum_{j\ge0}\hat{f}(j+k)
\overline{\hat{\phi}(j)}\Bigr|^2.
\end{equation}
Rather surprisingly, the formula \eqref{E:hbnorm} does \emph{not} extend to arbitrary $f\in\cH(b)$. 
It was shown in \cite{EFKMR16} that, for certain choices of $b,f$, it may happen that $\sum_{j\ge0}\hat{f}(j)\overline{\hat{\phi}(j)}$ diverges, and consequently that  \eqref{E:hbnorm} breaks down. Exactly the same argument, now used in conjunction with Theorem~\ref{T:stronger}, shows that this phenomenon may even occur when $b$ is outer.

\subsection{Summability methods}
It was shown in \cite{EFKMR16} that, 
if $\lim_{r\to1^-}\|f_r\|_{\cH(b)}=\infty$, then necessarily
the Taylor partial sums $s_n(f)$ of~$f$ 
and their Ces\`aro means $\sigma_n(f)$ satisfy
\begin{equation}\label{E:summability}
\limsup_{n\to\infty}\|s_n(f)\|_{\cH(b)}=\infty
\quad\text{and}\quad
\limsup_{n\to\infty}\|\sigma_n(f)\|_{\cH(b)}=\infty.
\end{equation}
Using Theorem~\ref{T:main}, we see that 
\eqref{E:summability} may occur even if $b$ is outer.
This raises the following question: given $b$ non-extreme in the unit ball of $H^\infty$, 
is there always a summability method $(S_n)$ (depending on $b$) 
such that $\lim_{n\to\infty}\|S_n(f)-f\|_{\cH(b)}=0$
for all $f\in\cH(b)$?

\bibliography{biblist}
\bibliographystyle{plain}

\end{document}